\documentclass[12pt]{article}

\usepackage{amsmath, amsfonts,amssymb,  amscd, amsthm, verbatim, mathtools, hyperref} %, amsxtra}
\usepackage{fullpage}
\usepackage{verbatim, colonequals}
\usepackage[all]{xy}
\usepackage{tikz,tkz-euclide}

\DeclareMathOperator{\fldchar}{char}

\DeclareMathOperator{\Frac}{Frac}

\DeclareMathOperator{\sing}{sing}
\DeclareMathOperator{\rk}{rk}
\DeclareMathOperator{\Jac}{Jac}

\def\A{\mathbb{A}}

\def\P{\mathbb{P}}

\def\F{\mathbb{F}}

\theoremstyle{plain}

\newtheorem{thm}{Theorem}
\newtheorem{Lem}[thm]{Lemma}
\newtheorem{Cor}[thm]{Corollary}

\theoremstyle{remark}
\newtheorem{Rem}[thm]{Remark}
\newtheorem{Que}[thm]{Question}

\renewcommand\footnotemark{}

\makeatletter
\newcommand{\subjclass}[2][2020]{%
  \let\@oldtitle\@title%
  \gdef\@title{\@oldtitle\footnotetext{#1 \emph{Mathematics subject classification.} #2}}%
}
\newcommand{\keyywords}[1]{%
  \let\@@oldtitle\@title%
  \gdef\@title{\@@oldtitle\footnotetext{\emph{Key words and phrases.} #1.}}%
}
\makeatother

\begin{document}

\title{Square values of several polynomials over a finite field}

\keyywords{square values, polynomials over a finite field, quadratic forms, Lang--Weil bound}
\author{Kaloyan Slavov\\[0.2em]}

\maketitle

\begin{abstract}
Let $f_1,\dots,f_m$ be polynomials in $n$ variables with coefficients in a finite field $\F_q$. We estimate the number of points 
$\underline{x}$ in $\F_q^n$ such that each value $f_i(\underline{x})$ is a nonzero square in $\F_q$. The error term is especially small when the $f_i$ define smooth projective quadrics with nonsingular intersections. We improve the error term in a recent work by Asgarli--Yip on mutual position of smooth quadrics. 
\end{abstract}

\section{Introduction}

Let $\F_q$ be a finite field with $q$ odd, and let 
$f\in\F_q[x_1,\dots,x_n]$ be a polynomial which is not a perfect square in $\overline{\F_q}[x_1,\dots,x_n]$. It is well-known that as 
$\underline{x}\in\F_q^n$ is chosen uniformly at random, the probability that $f(\underline{x})$ is a nonzero square (respectively, a non-square) in $\F_q$ is $1/2+O(q^{-1/2})$, where the implied constant depends only on $\deg f$. 

Consider now a finite field $k$ with $\fldchar k\neq 2$ and several polynomials $f_1,\dots,f_m\in k[x_1,\dots,x_n]$. For a finite field $\F_q\supset k$, as 
$\underline{x}\in\F_q^n$ is chosen uniformly at random, let $E_i$ (for $i=1,\dots,m$) be the event that $f_i(\underline{x})$ is a nonzero square in $\F_q$. We are interested in a necessary and sufficient condition on $f_1,\dots,f_m$ for the events $E_1,\dots,E_m$ to be ``independent" (up to an error). More precisely, 

\begin{Que}
Find a necessary and sufficient condition on $f_1,\dots,f_m$ for the following statement to hold: for a finite field $\F_q\supset k$ and a subset $S\subset\{1,\dots,m\}$, as $\underline{x}\in\F_q^n$ is chosen uniformly at random, the probability that $f_i(\underline{x})$ is a nonzero square in $\F_q$ for $i\in S$ and a non-square in $\F_q$ for $i\notin S$ is $1/2^m+O(q^{-1/2})$.  
\label{Q:1}
\end{Que}

\begin{Que} 
Under additional assumptions on $f_1,\dots,f_m$, improve the error bound $O(q^{-1/2})$.
\label{Q:2}
\end{Que}

The motivation for studying the above questions is a classical discrete geometry problem about mutual position of conics (see \cite{Asgarli_Yip} for references and latest results). Let $C$ be a smooth conic in the projective plane $\P^2$ over a finite field. In analogy with conics over the reals, one says that a point $P\in\P^2(\F_q)$ not on $C$ is external to $C$ if there exist two $\F_q$-tangent lines to $C$ passing through $P$, and internal to $C$ if there exist no $\F_q$-tangent lines to $C$ through $P$. Consider two distinct smooth conics $C$ and $D$. 
G. Korchm\'aros asks for an estimation of the number of points in $\P^2(\F_q)$ that are external to $C$ but internal to $D$. Asgarli--Yip \cite{Asgarli_Yip} consider more generally internal and external points to a smooth quadric $C=\{f_1=0\}$ in $\P^n$ with $n$ even and prove 
that a point $P\in\P^n(\F_q)$ not on $C$ is external (respectively, internal) to $C$ if and only if the value of a certain nonzero constant multiple of $f_1$ at $P$ is a square (respectively, a non-square) in $\F_q$. Thus, given two distinct smooth quadrics $C=\{f_1=0\}$ and $D=\{f_2=0\}$ in $\P^n$, the problem is to estimate the number of points $\underline{x}\in\P^n(\F_q)$ such that $f_1(\underline{x})$ is a nonzero square in 
$\F_q$ but $f_2(\underline{x})$ is a non-square in $\F_q$. To this end, Asgarli--Yip consider the hypersurface 
$\{f_1f_2=0\}$ in 
$\P^n$ and apply an advanced result of A. Rojas-Le\'on on estimates for singular multiplicative character sums. Our first main result Theorem \ref{thm_independent_events} below provides a simpler proof of \cite[Theorem 1.3]{Asgarli_Yip} and improves the implied constant in the error bound (see Corollary \ref{Cor:smooth_quadrics}). Moreover, Theorem
\ref{thm_improved_error_m_2} (which is the $m=2$ case of our second main result) improves significantly the error bound in the estimation under an additional assumption that $C\cap D$
is smooth (or has singular locus of dimension at most $n-3$).

Our approach to Question \ref{Q:1} is based on the Lang--Weil bound \cite{LW_bound} for the number of $\F_q$-points on the variety $X\subset\A^{n+m}$ defined as the vanishing of $f_i(\underline{x})-s_i^2$ for $i\in S$ and 
$f_i(\underline{x})-\nu s_i^2$ for $i\notin S$ (where $\nu\in\F_q$ is a non-square). We determine a condition for $X$ to be geometrically irreducible. 

For Question \ref{Q:2}, we consider $f_1,\dots,f_m\in k[x_0,\dots,x_n]$ homogeneous of degree $2$. We build a complete intersection $X\subset\P^{n+m}$ of dimension $n$. Under assumptions that $\{f_i=0\}$ and their intersections are nonsingular (or have low-dimensional singular loci), $X$ is nonsingular (or has a low-dimensional singular locus). Then we apply the bound of Hooley and Katz (\cite{Hooley}, \cite{Katz}) for $\#X(\F_q)$ with small error term. In contrast, the character sums approach of Asgarli--Yip is based on taking the product of the given polynomials $f_i$ from the onset: this creates a hypersurface whose singular locus is inevitably of large dimension, and
additional assumptions on the $f_i$ cannot improve the error bound. 

For a collection $\{g_i\}$ of polynomials (respectively, homogeneous polynomials), $V(\{g_i\})$ denotes the vanishing locus $\cap \{g_i=0\}$ in the appropriate affine (respectively, projective) space. 

\subsection{Answer to Question \ref{Q:1}}
It turns out that the obvious necessary condition (i) for (ii) in the theorem below to hold is also sufficient. 

\begin{thm} Let $n,m\geq 1$. Let $k$ be a finite field with 
$\fldchar k\neq 2$, and let 
$f_1,\dots,f_m\in k[x_1,\dots,x_n]$. %surely f_i\neq 0 in each implication anyways.  
The following are equivalent:

\begin{itemize}
\item[(i)] If $\varepsilon_1,\dots,\varepsilon_m\in\{0,1,-1\}$ and $\lambda\in k^*$ are such that $\lambda\prod_{i=1}^m f_i^{\varepsilon_i}$ is a square in $k(x_1,\dots,x_n)$, then 
$\varepsilon_1=\dots=\varepsilon_m=0$;

\item[(ii)] Let $\F_q\supset k$ be a finite field.
Let $S$ be a subset of $\{1,\dots,m\}$. Then the
number of $\underline{x}=(x_1,\dots,x_n)$ in $\F_q^n$ such that $f_i(\underline{x})$ is a nonzero square in $\F_q$ for $i\in S$ and a non-square in $\F_q$ for $i\notin S$ is
$q^n/2^m+O(q^{n-1/2}).$ 
\end{itemize}
\label{thm_independent_events}
\end{thm}

\begin{Rem}
When $m=2$, condition (i) above is simply that the square-free parts of $f_1$ and $f_2$ are both of degree at least $1$ and are not constant multiples of each other. Also, note that (i) holds whenever $f_1,\dots,f_m$ are irreducible and pairwise non-associate. On the other hand, (i) rules out, for example, a dependence of the form $f_1=g_2g_3$, $f_2=g_1g_3$, $f_3=g_1g_2$. However, in case of such a dependence, if we are interested say in the number of $\underline{x}$ such that all $f_1$, $f_2$, and $f_3$ assume nonzero square values at $\underline{x}$,  we can simply discard $f_3$ and apply the theorem just to the set 
$\{f_1,f_2\}$.  
\label{Rem:basic_cases}
\end{Rem}

\begin{Rem}
Suppose (i) and (ii) hold. The implied constant in the $O$-notation in (ii) depends only on $n$, $m$, and the degrees of $f_1,\dots, f_m$, but not on $q$. In fact, let  $d_i\colonequals\max(\deg(f_i),2)$, and set
$d\colonequals\prod_{i=1}^m d_i$. Let $S\subset\{1,\dots,m\}$, and let 
$N_{S}(f_1,\dots,f_m)$ be the number of points $\underline{x}$ in $\F_q^n$ such that $f_i(\underline{x})$ is a nonzero square in $\F_q$ for $i\in S$ and a non-square in $\F_q$ for $i\notin S$. 
An inspection of the proof of Theorem \ref{thm_independent_events} along with an application of   \cite[Theorem 7.1]{Cafure_Matera} shows that
\[\left|N_{S}(f_1,\dots,f_m)-\frac{q^n}{2^m}\right|\leq \frac{(d-1)(d-2)}{2^m}q^{n-1/2}+C_{d_1,\dots,d_m}q^{n-1}\]
for a constant $C_{d_1,\dots,d_m}$ that can be determined explicitly from the proof of Theorem \ref{thm_independent_events} and the results in \cite{Cafure_Matera}. 
\label{Rem:explicit_bound}
\end{Rem}

\begin{Rem}
In Theorem \ref{thm_independent_events}, we can add condition
(iii): There exists a finite field $k_1\supset k$ 
such that for every $i=1,\dots,m$ there exist 
$\underline{u},\underline{v}$ in $k_1^n$ such that
$f_i(\underline{u})f_i(\underline{v})$ is a non-square in $k_1$ and  $f_j(\underline{u})f_j(\underline{v})$ is a nonzero square in $k_1$ for all $j\neq i$. It is easy to see that (iii)$\implies$ (i) and (ii)$\implies$ (iii); thus, (iii) is also equivalent to (i) and (ii). 
\end{Rem}

\subsection{Improved error bound under additional assumptions}

Specialize from now on to the case when $f_1,\dots,f_m$ belong to the space $k[x_0,\dots,x_n]_2$ of homogeneous polynomials in $k[x_0,\dots,x_n]$ of degree $2$. For $f\in k[x_0,\dots,x_n]_2$ and 
$\underline{x}\in\P^n(\F_q)$, the notion of whether $f(\underline{x})$ is a nonzero square (respectively, a non-square) in $\F_q$ is well-defined. 

Since our second main result --- Theorem \ref{thm_improved_error} --- involves a number of parameters, we first state an illustrative special case. Set $\dim\emptyset=-1$.

\begin{thm} Let $k$ be a finite field with $\fldchar k\neq 2$. 
Let $f_1,f_2\in k[x_0,\dots,x_n]_2$ be irreducible and non-associate. Define
\[\sigma\colonequals \max\left(\dim V(f_1)_{\sing}, \dim V(f_2)_{\sing}, \dim V(f_1,f_2)_{\sing}\right).\]

Let $\F_q\supset k$ be a finite field, and let 
$S\subset\{1,2\}$. Then the number $N_S(f_1,f_2)$ of $\underline{x}$ in $\P^n(\F_q)$ such that $f_i(\underline{x})$ is a nonzero square in $\F_q$ for $i\in S$ and a non-square in $\F_q$ for $i\notin S$ satisfies
\[N_S(f_1,f_2)=\frac{q^n-q^{n-1}}{4}+O(q^{(n+\sigma+1)/2}).\]
\label{thm_improved_error_m_2}
\end{thm}

Consider $f_1,\dots,f_m\in k[x_0,\dots,x_n]_2$ that satisfy the assumption in Theorem \ref{thm_independent_events}(i). 

For any $1\leq r\leq m$ and any $r$-element subset $\{i_1,\dots,i_r\}$ of $\{1,\dots,m\}$, define
\[T(f_{i_1},\dots,f_{i_r})=\{\underline{x}\in V(f_{i_1},\dots,f_{i_r})\ |\ 
\rk\Jac(f_{i_1},\dots,f_{i_r})(\underline{x})<r
\}\]
and let $\sigma_{i_1,\dots,i_r}=\dim T(f_{i_1},\dots,f_{i_r})$. Notice that
\[
\sigma_{i_1,\dots,i_r}=\begin{cases}
\dim V(f_{i_1},\dots,f_{i_r})_{\sing},\quad & \text{if $V(f_{i_1},\dots,f_{i_r})$ has pure dimension $n-r$};\\
\dim V(f_{i_1},\dots,f_{i_r}),\quad &\text{if $\dim V(f_{i_1},\dots,f_{i_r})>n-r$.}
\end{cases}
\]

Set $\sigma=\max(\sigma_{i_1,\dots,i_r}),$ with maximum over all $1\leq r\leq m$ and all $r$-element subsets $\{i_1,\dots,i_r\}$ of $\{1,\dots,m\}$. 
 
Let $l=\max(n-m-\sigma-1,0)$. Note that $l\leq n-1$. Moreover, if $l\geq 1$, then $n-m\geq \sigma+2\geq 1$, so $m\leq n-1$. 

For $i\geq 0$, denote
\[\pi_i(q)=\#\P^i(\F_q)=q^i+\dots+q+1.\]

\begin{thm} Let $k$ be a finite field with $\fldchar k\neq 2$. Let $f_1,\dots,f_m\in k[x_0,\dots,x_n]_2$ satisfy the assumption in Theorem \ref{thm_independent_events}(i). 
Define $\sigma$ and $l$ as above. Suppose that for any $1\leq r\leq\min(l+1,m)$ and any $r$-element subset $\{f_{i_1},\dots,f_{i_r}\}$ of $\{1,\dots,m\}$, we have $V(f_{i_1},\dots,f_{i_{r-1}})\not\subset V(f_{i_r})$.

Let $\F_q\supset k$ be a finite field, and let 
$S\subset\{1,\dots,m\}$. Then the number $N_S(f_1,\dots,f_m)$ of $\underline{x}$ in $\P^n(\F_q)$ such that $f_i(\underline{x})$ is a nonzero square in $\F_q$ for $i\in S$ and a non-square in $\F_q$ for $i\notin S$ satisfies
\begin{equation}
N_S(f_1,\dots,f_m)=\frac{1}{2^m}\sum_{0\leq r\leq \min(m,l)}(-1)^r\binom{m}{r}\pi_{n-r}(q)+O(q^{\gamma}),
\label{eq:main_thm_formula_error}
\end{equation}
where
\[\gamma=\max\left(\frac{n+\sigma+1}{2},n-l-1\right).\]
\label{thm_improved_error}
\end{thm}

\subsection{An application to mutual position of quadrics} 

Theorem \ref{thm_independent_events} and Remarks \ref{Rem:basic_cases}, \ref{Rem:explicit_bound}, combined with 
\cite[Lemma 2.4]{Asgarli_Yip}
imply

\begin{Cor} Let $C$ and $D$ be distinct smooth quadrics in $\P^n$ (with $n$ even). Then the number $N$ of points 
$\underline{x}\in\P^n(\F_q)$ external to $C$ but internal to $D$ satisfies
\[\left|N-\frac{q^n}{4}\right|\leq \frac{3}{2}q^{n-1/2}+Cq^{n-1},\]
with an effectively computable constant $C$ depending only on $n$. 
\label{Cor:smooth_quadrics}
\end{Cor}

Here the factor $3/2$ of $q^{n-1/2}$ improves the $3.8^{n+1}$ in \cite{Asgarli_Yip} (see lemmas 3.1 and 3.2 in \cite{Asgarli_Yip}). 

The improved error bound in Theorem \ref{thm_improved_error_m_2} gives

\begin{Cor}
Let $C$ and $D$ be distinct smooth quadrics in $\P^n$ (with $n$ even). Let $\sigma\colonequals\dim(C\cap D)_{\sing}$.  Then the number of points 
$\underline{x}\in\P^n(\F_q)$ external to $C$ but internal to $D$ equals
$(q^n-q^{n-1})/4+O(q^{(n+\sigma+1)/2})$.
\end{Cor}

\section{Proofs of the results.} 

The proof of Theorem \ref{thm_independent_events} is based on the following

\begin{Lem} Let $k$ be a field, and let $f_1,\dots,f_m\in k[x_1,\dots,x_n]$. Suppose that for 
$\varepsilon_1,\dots,\varepsilon_m\in\{0,1,-1\}$, the product $\prod_{i=1}^m f_i^{\varepsilon_i}$ is a square in $k(x_1,\dots,x_n)$ only when $\varepsilon_1=\dots=\varepsilon_m=0$. Then the $k[x_1,\dots,x_n]$-algebra
\[k[x_1,\dots,x_n,s_1,\dots,s_m]/(f_1-s_1^2,\dots,f_m-s_m^2)\]
is an integral domain. 
\label{Lem_integral_domain}
\end{Lem}

\begin{proof}
Set $\underline{x}=(x_1,\dots,x_n)$ for brevity. Let $R_0=k[\underline{x}]$. For $i=1,\dots,m$, define
\[R_i\colonequals k[\underline{x},s_1,\dots,s_i]/(s_1^2-f_1(\underline{x}),\dots,s_i^2-f_i(\underline{x})).\]

We prove by induction on $i\in\{0,\dots,m\}$ that $R_i$ is an integral domain and for every $\varepsilon_{i+1},\dots,\varepsilon_m\in\{0, 1,-1\}$,
\[\text{if}\quad\prod_{j=i+1}^m f_j^{\varepsilon_j}\quad\text{is a square in $\Frac(R_i)$, then}\
 \varepsilon_{i+1}=\dots=\varepsilon_m=0.\] 

This clearly holds for $i=0$. Suppose that $1\leq i\leq m$ and the statement above holds for $i-1$. 
Note that
\[R_i\simeq R_{i-1}\oplus R_{i-1}s_i\qquad\text{as $R_{i-1}$-modules}.\]  
Suppose that $(a+bs_i)(c+ds_i)=0$ in $R_i$, with $a,b,c,d\in R_{i-1}$, $(a,b)\neq (0,0)$, $(c,d)\neq (0,0)$. Then $ac+bdf_i=0$ and $ad+bc=0$ in $R_{i-1}$. Therefore 
$c^2-d^2f_i=0$. Then $f_i$ would be a square in $\Frac(R_{i-1})$, contrary to the inductive hypothesis. Thus $R_i$ is an integral domain. 

Since $f_i$ is not a square in 
$F_{i-1}\colonequals\Frac(R_{i-1})$, the $F_{i-1}$-algebra
$F_i\colonequals F_{i-1}[s_i]/(s_i^2-f_i)$ is a field; the map $R_i\to F_i$ then identifies $F_i$ with $\Frac(R_i)$. Note that $F_i\simeq F_{i-1}\oplus F_{i-1}s_i$ as $F_{i-1}$-vector spaces.

Finally, suppose that for some $\varepsilon_{i+1},\dots,\varepsilon_m\in\{0,1,-1\}$, we have that $\prod_{j=i+1}^m f_j^{\varepsilon_j}$ is a square in $\Frac(R_i)$. Then we can write 
\[\prod_{j=i+1}^m f_j^{\varepsilon_j}=\left(\frac{u}{v}+\frac{u_1}{v_1}s_i\right)^2=\frac{u^2}{v^2}+\frac{u_1^2}{v_1^2}f_i+2\frac{uu_1}{vv_1}s_i\quad\text{for some $u,v,u_1,v_1\in R_{i-1}$ with $v,v_1\neq 0$.}\]
Since LHS belongs to $F_{i-1}$, we deduce $u=0$ or $u_1=0$. 
If $u=0$, then with $\varepsilon_i=-1$, we would have that $\prod_{j=i}^m f_j^{\varepsilon_j}=(u_1/v_1)^2$ is a square in $F_{i-1}$, which contradicts the inductive hypothesis. Therefore in fact
$u_1=0$. Now $\prod_{j=i+1}^m f_j^{\varepsilon_j}=(u/v)^2$ is a square in $\Frac(R_{i-1})$; hence, indeed, $\varepsilon_{i+1}=\dots=\varepsilon_m=0$ by the inductive hypothesis. 
\end{proof}

\begin{proof}[Proof of Theorem \ref{thm_independent_events}] We first prove (ii)$\implies$(i). Suppose (ii) holds. Suppose $\lambda\prod_{i=1}^m f_i^{\varepsilon_i}$ is a square in $k(x_1,\dots,x_n)$ for some $\varepsilon_i\in\{0,1,-1\}$ and some $\lambda\in k^*$. Suppose that $\varepsilon_i\neq 0$ for some $i$. 

Using (ii), we can find a finite field $\F_q\supset k$ such that $\lambda$ is a square in $\F_q$ and such that there exists an $\underline{x}\in\F_q^n$ such that $f_i(\underline{x})$ is a non-square in $\F_q$ while $f_j(\underline{x})$ is a nonzero square in $\F_q$ for each $j\neq i$. Then $\lambda\prod_{j=1}^m f_j(\underline{x})^{\varepsilon_j}$ would be a non-square in 
$\F_q$, which is a contradiction. 

We now prove (i)$\implies$(ii). Suppose that (i) holds.
Then $f_i\neq 0$ for each $i=1,\dots,m$. 
Let $\F_q\supset k$ be a finite field, and let 
$S\subset\{1,\dots,m\}$. 

Pick a non-square $\nu\in\F_q$. For $i=1,\dots,m$, consider the polynomial $g_i$ in $\F_q[\underline{x},s_1,\dots,s_m]$ defined by
\[
g_i=
\begin{cases}
f_i(\underline{x})-s_i^2\quad &\text{if}\ i\in S;\\
f_i(\underline{x})-\nu s_i^2\quad &\text{if}\ i\notin S.
\end{cases}
\]
Consider the $\F_q$-variety \[X\colonequals \{g_1=0,\dots,g_m=0\}\subset\A^n_{x_1,\dots,x_n}\times\A^m_{s_1,\dots,s_m};\]
it comes with a projection $\varphi\colon X\to \A^n_{x_1,\dots x_n}$. Note that $\varphi$ is finite and surjective. In particular, $\dim X=n$.

We claim that $X$ is geometrically irreducible. Over $\overline{\F_q}$, note that $X$ becomes isomorphic to 
$\{f_1(\underline{x})-s_1^2=0,\dots,f_m(\underline{x})-s_m^2=0\}$. Thus, we have to establish that the $\overline{\F_q}$-algebra
\[\overline{\F_q}[\underline{x},s_1,\dots,s_m]/(f_1(\underline{x})-s_1^2,\dots,f_m(\underline{x})-s_m^2)\]
is an integral domain. We check that the hypothesis in 
Lemma \ref{Lem_integral_domain} is satisfied for the field $\overline{\F_q}$ and the polynomials $f_1,\dots,f_m$. Indeed, suppose that $\beta\colonequals\prod f_i^{\varepsilon_i}$ is a square in $\overline{\F_q}(x_1,\dots,x_n$) for some $\varepsilon_1,\dots,\varepsilon_m\in\{0,1,-1\}$. By Lemma \ref{lemma_k_to_k_bar_condition} below, there exists a
$\lambda\in k^*$ such that $\lambda\beta$ is a square in $k(x_1,\dots,x_n)$. But then (i) implies $\varepsilon_1=\dots=\varepsilon_m=0$. 

The restriction of $\varphi$ to $X-\cup_{i=1}^m V(s_i)$ is $2^m:1$ on $\F_q$-points. Therefore the number $N_S(f_1,\dots,f_m)$ of $\underline{x}\in\F_q^n$ such that $f_i(\underline{x})$ is a nonzero square in $\F_q$ for $i\in S$ and a non-square in $\F_q$ for $i\notin S$ satisfies
\[
N_S(f_1,\dots,f_m) =\frac{1}{2^m}\#\left(X-\bigcup_{i=1}^m V(s_i)\right)(\F_q)\]
For each $i=1,\dots,m$, note that 
$X\cap V(s_i)$ is a proper closed subset of $X$ and therefore has dimension $n-1$. Thus $\#(X\cap V(s_i))(\F_q)=O(q^{n-1})$.  

The Lang--Weil bound \cite{LW_bound} (or the version 
\cite[Theorem 7.1]{Cafure_Matera} with an explicit error term) implies
\[\#X(\F_q)=q^n+O(q^{n-1/2}). \]
Therefore
\[N_S(f_1,\dots,f_m)=\frac{1}{2^m}\#X(\F_q)+O(q^{n-1})=\frac{q^n}{2^m}+O(q^{n-1/2}).\qedhere\]
\end{proof}

\begin{Lem} Let $k$ be a finite field, and let $\beta\in k(x_1,\dots,x_n)$. Suppose that $\beta$ is a square in 
$\overline{k}(x_1,\dots,k_n)$. Then there exists a $\lambda\in k^*$ such that $\lambda\beta$ is a square in 
$k(x_1,\dots,x_n)$. 
\label{lemma_k_to_k_bar_condition}
\end{Lem}

\begin{proof}
Write $\beta=f/g$ with $f,g\in k[x_1,\dots,x_n]$. 
We can assume that $f$ and $g$ are square-free in the UFD $k[x_1,\dots,x_n]$. We claim that $f$ and $g$ remain square-free in the UFD $\overline{k}[x_1,\dots,x_n]$. Indeed, the $k$-algebra $A\colonequals k[x_1,\dots,x_n]/(f)$ is reduced; since $k$ is a perfect field, the $\overline{k}$-algebra $A\otimes_k\overline{k}$ is also reduced, or, equivalently, $f$ is square-free in $\overline{k}[x_1,\dots,x_n]$. Similarly for $g$.    

By assumption, $f/g=u^2/v^2$ in $\overline{k}(x_1,\dots,x_n)$, for some $u,v\in\overline{k}[x_1,\dots,x_n]$; in other words, $fv^2=gu^2$ in $\overline{k}[x_1,\dots,x_n]$. The square-free part of an element in a UFD is unique up to scaling by a unit. Therefore $f=\lambda g$ for some $\lambda\in\overline{k}$. But then $\lambda=f/g$ belongs to $\overline{k}\cap k(x_1,\dots,x_n)=k$. In particular, $\lambda\beta=\lambda^2$ is a square in $k(x_1,\dots,x_n)$. \end{proof}

\begin{proof}[Proof of Theorem \ref{thm_improved_error}.]
Let $\nu\in\F_q$ be a non-square.  Consider
\[
g_i=
\begin{cases}
f_i(x_0,\dots,x_n)-s_i^2\quad &\text{if}\ i\in S;\\
f_i(x_0,\dots,x_n)-\nu s_i^2\quad &\text{if}\ i\notin S.
\end{cases}
\]
Consider the projective variety $X=V(g_1,\dots,g_m)\subset\P^{n+m}_{[x_0\colon\dots\colon x_n\colon s_1\colon\dots\colon s_m]}$. As in the proof of Theorem \ref{thm_independent_events}, $X$ is geometrically irreducible and comes with a finite surjective morphism 
$\varphi\colon X\to\P^n_{[x_0\colon\dots\colon x_n]}$ of degree $2^m$. In particular, $\dim X=n$ and $X$ is a complete intersection in $\P^{n+m}$. The restriction of $\varphi$ to
$X-\bigcup_{i=1}^m V(s_i)$ is $2^m\colon 1$ on $\F_q$-points. Therefore
\begin{equation}
N_S(f_1,\dots,f_m)=\frac{1}{2^m}\#\left(X-\bigcup_{i=1}^m V(s_i)\right)(\F_q).
\label{eq:fraction_for_N_S}
\end{equation}

We will compute the right-hand side of (\ref{eq:fraction_for_N_S}) by using the Inclusion-Exclusion Principle. For $0\leq r\leq m$ and an $r$-element subset $\{i_1,\dots,i_r\}$ of $\{1,\dots,m\}$, set
$X_{i_1,\dots,i_r}=X\cap V(s_{i_1},\dots,s_{i_r})$. 

\begin{Lem}
We have $\dim (X_{\sing})\leq \sigma$.
\end{Lem}

\begin{proof}
Notice that
\[\Jac(g_1,\dots,g_m)=(\Jac(f_1,\dots,f_m)\ |\   D),\]
where $D$ is the diagonal matrix whose $(i,i)$-entry is $-2s_i$ if $i\in S$ or $-2\nu s_i$ if $i\notin S$.
For $0\leq r\leq m$ and an $r$-element subset $\{i_1,\dots,i_r\}$ of $\{1,\dots,m\}$, set 
\[X_{i_1,\dots,i_r}^+=X_{i_1,\dots,i_r}\cap \{s_j\neq 0\ \text{for each}\ j\notin\{i_1,\dots,i_r\}\}.\]
Then $X$ is the union of the locally closed $X_{i_1,\dots,i_r}^+$, and we need to establish that $X_{\sing}\cap X_{i_1,\dots,i_r}^+$ has dimension at most $\sigma$. The Jacobian criterion for $X$ implies 
\[X_{\sing}\cap X_{i_1,\dots,i_r}^+\subset\varphi^{-1}(T_{i_1,\dots,i_r});\]
since $\varphi$ is a finite map and $\dim T_{i_1,\dots,i_r}\leq\sigma$, the conclusion follows. 
\end{proof}

\begin{Lem}
Let $0\leq r\leq \min(m,l)$, and let $\{i_1,\dots,i_r\}$ be an $r$-element subset of $\{1,\dots,m\}$. 
Then 
$X_{i_1,\dots,i_r}$ is geometrically irreducible of dimension $n-r$ (in particular, it is a complete intersection in 
$\P^{n+m}$) and $\dim (X_{i_1,\dots,i_r})_{\sing}\leq \sigma$. For $l+1\leq r\leq m$, we have $\dim X_{i_1,\dots,i_r}\leq n-l-1$.
\end{Lem} 

\begin{proof}
The case $l=0$ is clear, so suppose $l=n-m-\sigma-1>0$. 
We prove by induction on $r$ that for any $r$-element subset 
$\{i_1,\dots,i_r\}$ of $\{1,\dots,m\}$, the following hold in turn:
(1) $X_{i_1,\dots,i_r}$ has pure dimension $n-r$; (2) 
$\dim (X_{i_1,\dots,i_r})_{\sing}\leq \sigma$;  
(3) $X_{i_1,\dots,i_r}$ is geometrically irreducible. 
These hold for $r=0$. Suppose $1\leq r\leq \min(m,l)$ and (1)--(3) hold for $r-1$. Consider $\{i_1,\dots,i_r\}$ of cardinality $r$. 
First note that
$X_{i_1,\dots,i_r}\subsetneq X_{i_1,\dots,i_{r-1}}$ since there exists a point $P\in\P^n$ such that $f_{i_r}(P)\neq 0$ while $f_j(P)=0$ for $j\in\{i_1,\dots,i_{r-1}\}$.
Therefore $X_{i_1,\dots,i_r}=X_{i_1,\dots,i_{r-1}}\cap V(s_{i_r})$ has pure dimension $n-r$. Next, 
view $X_{i_1,\dots,i_r}$ as a complete intersection in $\P^{n+m-r}$ defined by $m$ equations, namely $f_j(x_0,\dots,x_n)=0$ for $j\in\{i_1,\dots,i_r\}$ and $g_j=0$ for $j\notin\{i_1,\dots,i_r\}$.  Write down the Jacobian matrix for these equations and note that
$(X_{i_1,\dots,i_r})_{\sing}$ identifies with 
$X_{\sing}\cap V(s_{i_1},\dots,s_{i_r})$; therefore,
$\dim (X_{i_1,\dots,i_r})_{\sing}\leq \dim X_{\sing}\leq\sigma$. Finally, suppose $Y,Z$ are two irreducible components of $X_{i_1,\dots,i_{r}}$ over $\overline{\F_q}$. Each of them has dimension $n-r$. Then $\dim (Y\cap Z)\geq n-m-r$. But $Y\cap Z$ is contained in  $(X_{i_1,\dots,i_r})_{\sing}$. Therefore $n-m-r\leq \sigma$, giving $r\geq n-m-\sigma=l+1$. This contradiction establishes (3). For $r=l+1$, the part of the argument concerning the dimension drop still applies, establishing the second part of the statement.    
\end{proof}

Let $0\leq r\leq m$ and let $\{i_1,\dots,i_r\}$ be an $r$-element subset of $\{1,\dots,m\}$. If $r\geq l+1$, then $\dim X_{i_1,\dots,i_r}\leq n-l-1$ and $\#X_{i_1,\dots,i_r}(\F_q)=O(q^{n-l-1})$ is absorbed into the error term $O(q^{\gamma})$ in (\ref{eq:main_thm_formula_error}). For $0\leq r\leq\min(m,l)$, we can apply the bound of Hooley and Katz (\cite{Hooley}, \cite{Katz}) 
\[\#X_{i_1,\dots,i_r}(\F_q)=\pi_{n-r}(q)+O(q^{(n-r+\sigma+1)/2}).\]
Substituting these into the Inclusion-Exclusion expansion of (\ref{eq:fraction_for_N_S})
gives (\ref{eq:main_thm_formula_error}), since the error terms  are absorbed into the error term 
$O(q^{\gamma})$ (note that $(n-r+\sigma+1)/2\leq (n+\sigma+1)/2\leq\gamma$). 
\end{proof}


\begin{thebibliography}{99}

\bibitem{Asgarli_Yip} S. Asgarli, C. H. Yip, Mutual position of two smooth quadrics over finite fields, \url{https://arxiv.org/pdf/2404.06754}.
\bibitem{Cafure_Matera} A. Cafure, G. Matera, Improved explicit estimates on the number of solutions of equations over a finite field,
Finite Fields and Their Applications 12 (2006) 155--185.
\bibitem{Hooley} C. Hooley, On the number of points on a complete intersection over a finite field, J. Number Theory 38 (1991), 338--358. 
\bibitem{Katz} N. M. Katz, Number of points of singular complete intersections, Appendix to \cite{Hooley}, 355--358.
\bibitem{LW_bound} S. Lang, A. Weil, Number of points of varieties in finite fields. Amer.
J. Math., 76 (1954) 819–827.
\end{thebibliography}
\end{document}